\documentclass[a4paper]{amsart} 
\usepackage{amsmath}

\usepackage{amssymb}
\usepackage{verbatim}
\usepackage{amsthm}
\usepackage{mathrsfs}
\usepackage{graphicx}
\usepackage{mathtools}
\usepackage[colorlinks,pdfdisplaydoctitle,linkcolor=blue,citecolor=red]{hyperref}
\usepackage{caption}
\usepackage{subcaption}
\usepackage{url}
\usepackage{epstopdf}
\usepackage{pgf,tikz}
\usepackage{bbm}

\renewcommand{\Delta}{\triangle}

\definecolor{darkblue}{rgb}{0,0,0.7}

\definecolor{darkgreen}{rgb}{0.01,0.75,0.24}

\def \Ee[#1]{\mathcal{E}^{\text{{#1}}}}
\def\R{\mathbf{R}}

\def\pa[#1,#2]{\frac{\partial {#1}}{\partial {#2}} }

\def\idom[#1,#2,#3]{\int_{#1}\hspace{1pt} {#2} \hspace{1pt} \text{d}{#3}}
\def\res[#1,#2]{\left.{#1}\right|_{#2}}

\def \div {\text{div}}

\def\gt{\rightarrow}

\def\var[#1,#2]{\langle \delta \mathcal{E}^{\text{{#1}}}({#2}),v\rangle}
\def\vars[#1,#2,#3]{\langle \delta^2\mathcal{E}^{\text{{#1}}}({#2})v,{#3}\rangle}
\def\vard[#1,#2,#3,#4]{\langle \delta\mathcal{E}^{\text{{#1}}}({#2})-\delta\mathcal{E}^{\text{{#3}}}({#4}),v\rangle}

\def\F{\mathcal{F}}
\def\P{\mathcal{P}}

\def\Rc{\mathcal{R}}

\def\N{\mathbb{N}}

\def\x{\mathbf{x}}

\newcommand{\A}{\mathcal{A}}

\newcommand{\W}{\mathcal{W}}

\newcommand{\wgt}{\rightharpoonup}

\newcommand{\be}{\begin{equation}}
\newcommand{\en}{\end{equation}}
\newcommand{\ben}{\begin{equation*}}
\newcommand{\enn}{\end{equation*}}
\newcommand{\bea}{\begin{aligned}}
\newcommand{\ena}{\end{aligned}}

\def\ba#1\ena{\begin{align}#1\end{align}}
\def\ban#1\enan{\begin{align*}#1\end{align*}}

\theoremstyle{plain}
\newtheorem{theorem}{Theorem}[section]
\newtheorem{definition}[theorem]{Definition}
\newtheorem{lemma}[theorem]{Lemma}

\newtheorem{assumption}[theorem]{Assumptions}

\theoremstyle{remark}
\newtheorem{remark}[theorem]{Remark}

\numberwithin{equation}{section}

\begin{document}
\title[fractional kinetic Fokker-Planck]{An operator splitting scheme for the fractional kinetic Fokker-Planck equation}

\begin{abstract}
In this paper, we develop an operator splitting scheme for the fractional kinetic Fokker-Planck equation (FKFPE).
The scheme consists of two phases: a fractional diffusion phase and a kinetic transport phase. 
The first phase is solved exactly using the convolution operator while the second one is solved approximately
using a variational scheme that minimizes an energy functional with respect to a certain Kantorovich optimal transport cost functional.
We prove the convergence of the scheme to a weak solution to  FKFPE.
As a by-product of our analysis, we also establish a variational formulation for a  kinetic transport equation that is relevant in the second phase. 
Finally, we discuss some extensions of our analysis to more complex systems.
\end{abstract}

\author[M. H. Duong]{Manh Hong Duong }
\address[M. H. Duong]{Department of Mathematics, Imperial College London, London SW7 2AZ, UK}
\email{m.duong@imperial.ac.uk}

\author[Y. Lu]{Yulong Lu}
\address[Y. Lu]{Department of Mathematics, Duke University, Durham NC 27708, USA}
\email{yulonglu@math.duke.edu}

\keywords{Operator splitting methods, variational methods, fractional kinetic Fokker-Planck equation, kinetic transport equation, optimal transportation.}
\subjclass[2010]{Primary: 49S05, 35Q84; Secondary: 49J40.}

\maketitle

\section{Introduction}

In this paper, we study the existence of solutions to the following fractional kinetic Fokker-Planck equation (FKFPE)
\begin{equation}
\begin{cases}
\partial_t f+v\cdot\nabla_x f=\div_v(\nabla \Psi(v)f)-(-\Delta_v)^s f \quad \text{in} \quad \R^d\times \R^d\times (0,\infty),\\
f(x,v,0)=f_0(x,v)\quad \text{in} \quad \R^d\times \R^d,
\end{cases}
\label{eq:fracKramers}
\end{equation}
with $s\in (0,1]$. In the above, $\mathrm{div}$ denotes the divergence operator; the differential operators $\nabla, \div$ and $\Delta$ with subscripts $x $  and $v$ indicate that these operators act only on the corresponding variables; the operator $-(-\Delta_v)^s$ is the fractional Laplacian operator on the variable $v$, where the fractional Laplacian $-(-\Delta)^{s}$, is defined by
$$
-(-\Delta)^s f (x) := -\mathcal{F}^{-1} (|\xi|^{2s}  \mathcal{F}[f] (\xi)) (x).
$$
Here $\mathcal{F}$ denotes the Fourier transform on $\R^d$, i.e. $\F [f] (\xi) = \frac{1}{(2\pi)^{d/2}}\int_{\R^d} f(x) e^{-ix\cdot \xi}dx$. 
Note that the fractional Laplacian operator with $0<s<1$ is a non-local operator since it can also be expressed as the singular integral 
$$
-(-\Delta)^s f (x) = -C_{d, s} \int_{\R^d} \frac{f(x)  - f(y)}{|x - y|^{d+2s}} dy, 
$$
where the normalisation constant is given by
$
C_{d, s} = s 2^{2s} \Gamma(\frac{d + 2s}{2})/ (\pi^{\frac{d}{2}} \Gamma(1-s)) 
$ 
and $\Gamma(t)$ is the Gamma function. See \cite{kwasnicki2017ten}  for more equivalent definitions of fractional Laplacian operator.

The equation \eqref{eq:fracKramers} is interesting to us because it can be viewed as the Fokker-Planck (forward Kolmogorov) equation of the following generalized Langevin equation
\begin{equation}
\label{SDE}
\bea
& \frac{d X_t}{d t} = V_t,\\
& \frac{d V_t}{d t} = - \nabla \Psi (V_t) + L_t^s,
\ena
\end{equation}
where $L_t^s$ is the L\'evy stable process with exponent $2s$. The stochastic differential equation (SDE) \eqref{SDE} describes
the motion of a particle moving under the influence of a (generalized) frictional force and a stochastic noise and in the absence of an external force field. 
 FKFPE \eqref{eq:fracKramers} is the evolution of the probability distribution of $(X_t,V_t)$. 
In particular, the fractional operator $-(-\Delta)^s$ is the Markov generator of the process $L_t^s$. 
When $s=1$ and $\Psi(v)=\frac{|v|^2}{2}$, equation \eqref{eq:fracKramers} becomes the classical kinetic Fokker-Planck (or Kramers) 
equation (without external force field)
which is a local PDE and has been used widely in chemistry as a simplified model for chemical reactions~~\cite{Kramers40,HTB90} and
in statistical mechanics~\cite{Nelson1967,Risken}.  The non-local L\'evy process plays an important role in modelling systems that include jumps 
and long-distance interactions such as anomalous diffusion or transport in confined plasma~\cite{Applebaum2009}. Singular limits of Equation~\eqref{eq:fracKramers} with $\Psi(v)=\frac{|v|^2}{2}$ was studied in~\cite{Cesbron2012}, see also~\cite{Cesbron2017} for a similar result for the same equation but on a spatially bounded domain. In a recent work~~\cite{SanchezCesbron2016TMP}, the authors have extended~\cite{Cesbron2012} to a system that contains an additional external force field and they have also proved its well-posedness by the means of the Lax-Milgram theorem. We will prove the existence of solutions of \eqref{eq:fracKramers} for a general $\Psi$ based 
on the trick of operator splitting.  For more recent developments on PDEs involving the fractional Laplacian operator, we refer the interested reader to 
expository surveys~\cite{Vazquez2017, Vazquez2014,Vazquez2012}.

The aim of this paper is to develop a variational formulation for approximating solutions to equation~\eqref{eq:fracKramers}. 
The theory of variational formulation for PDEs took off with the introduction of Wasserstein gradient flows by the seminal work of Jordan, Kinderlehrer and Otto~\cite{JKO98}. 
Such a variational structure has important applications for the analysis of an evolution equation such as providing general methods for proving well-posedness~\cite{AGS08}
and characterizing large time behaviour~(e.g.,~\cite{CarrilloMcCannVillani03}), giving rise to natural numerical discretizations (e.g.,~\cite{DuringMatthesMilisic10}),
and offering techniques for the analysis of singular limits (e.g.,~\cite{SandierSerfaty04, Stefanelli08, AMPSV12, DLPS17}).
There are now a significantly large number of papers in exploring variational structures for local PDEs, see the aforementioned papers and references therein
as well as the monographs~\cite{AGS08, Vil03} for more details. However,  variational formulations for non-local PDEs are less understood.  Erbar~\cite{erbar2014} showed that the fractional heat equation is
a gradient flow of the Boltzmann entropy with respect to a new modified Wasserstein distance that is built from the L\'evy measure and based on
the Benamou-Brenier variant of the Wasserstein distance. Bowles and Agueh~ \cite{AguelBowles2015} proved the existence of the fractional Fokker-Planck equation
\be\label{eq:FFP}
\begin{cases}
\partial_t f =\div_v(\nabla \Psi(v)f)-(-\Delta_v)^s f \quad \text{in} \quad \R^d\times (0,\infty),\\
f(v,0)=f_0(v)\quad \text{in} \quad  \R^d,
\end{cases}
\en
which can be viewed as the spatially homogeneous version of equation~\eqref{eq:fracKramers} or the fractional heat equation with a drift. 
Erbar's proof is variational based on the so-called ``evolution variational inequality" concept introduced in~\cite{AGS08}. 
However, it seems that his method can not be extended to the fractional Fokker-Planck equation since the distance 
that he introduced was particularly tailored for the Boltzmann entropy. Instead, Bowles and Agueh's proof is ``semi-variational" based on a novel splitting argument which we sketch now.
They split up the original dynamics \eqref{eq:FFP} into two processes: 
a fractional diffusion process, namely $\partial_t f=-(-\Delta)^s f$, and a transport process in the field of the potential $\Psi$, namely $\partial_t f=\div(\nabla \Psi f)$, 
and then alternatively run these processes on a small time interval.
Furthermore, the transport process can be understood as a Wasserstein gradient flow of the potential energy.
By adopting a suitable interpolation of the individual processes, they were able to show that the constructed splitting scheme converges to a weak solution of \eqref{eq:FFP}. 
In the literature, the technique of operator splitting is often used to construct numerical methods for solving PDEs, see \cite{HKLR10}. On the theoretical side,
 the idea of splitting had also been used to study the well-posedness of  PDEs, see \cite{CG04, Agueh16} on kinetic equations and \cite{Alibaud07, DGV03} on fractional PDEs.  
 
 In the present work, we adopt the same splitting argument in \cite{AguelBowles2015} to construct a weak solution to the fractional kinetic equation \eqref{eq:fracKramers}. More specifically, we split the dynamics described in \eqref{eq:fracKramers} by two phases: 
 
 \begin{enumerate}
 \item Fractional diffusion phase. At every fixed position $x\in \R^d$, the probability density $f(x, v, t)$, as a function of velocity $v$, evolves according to the fractional heat equation
\be\label{eq:fhe}
\partial_t f = - (-\Delta_v)^{s} f.
\en

\item Kinetic transport phase. The density $f(x,v, t)$ evolves according to the following equation 
\be\label{eq:kt}
\partial_t f + v \cdot \nabla_x f = \mathrm{div}_v(\nabla \Psi(v)f).
\en
 \end{enumerate}
We expect that successive alternative iterating the above two phases with vanishing period of time would give an approximation to the dynamics \eqref{eq:fracKramers}. The key difference between our splitting scheme above and the scheme in \cite{AguelBowles2015} is that the transport process here is not only driven by the potential energy but also the kinetic energy. In~\cite{AguelBowles2015}, the transport process is approximated by a discrete Wasserstein gradient flow based on the work~\cite{KinderlehrerTudorascu06}. However, due to the presence of the kinetic term, the kinetic transport equation is not a Wasserstein gradient flow; thus one can no longer use the Wasserstein distance.  To overcome this obstacle, we employ instead the minimal acceleration cost function and the associated Kantorovich optimal transportation cost functional that has been used in~\cite{Hua00, DPZ13a} for the kinetic Fokker-Planck equation and in~\cite{GW09} for the isentropic Euler system, see Section~\ref{sec: kinetic transport}.
  
\subsection{Main result}
Throughout the paper, we make the following important assumption on the potential $\Psi$.

\begin{assumption}\label{ass:psi}
$\Psi$ is non-negative and $\Psi \in C^{1,1}\cap C^{2,1}(\R^d)$. 
\end{assumption}
We adopt the following notion of weak solution to KFPE \eqref{eq:fracKramers}.
\begin{definition}\label{def:weak}
Let $f_0$ be a non-negative function such that $f_0\in \P^2_a(\R^{2d}) \cap L^p(\R^{2d})$ for some $1 < p \leq \infty$ and $\int_{\R^{2d}} f_0(x, v) \Psi(v)dvdx < \infty$. We say that $f(x,v, t)$ is a weak solution to \eqref{eq:fracKramers} if it satisfies the following:
\begin{enumerate}
\item $\int_{\R^{2d}} f(x, v, t) dxdv = \int_{\R^{2d}} f_0(x, v) dxdv = 1$ for any $t\in(0,T)$. 

\item $f(x,v, t) \geq 0$ for a.e. $(x,v,t) \in \R^{2d}\times (0, T)$. 

\item For any test function $\varphi\in C_c^\infty(\R^{2d} \times (-T, T))$, 
$$
\bea
& \int_{0}^T \int_{\R^{2d}}f(x, v, t) (\partial_t \varphi + v \cdot \partial_x \varphi + \nabla_v \Psi \cdot \nabla_v \varphi
- (-\Delta_v)^{s} \varphi) dt dx dv \\
& \qquad + \int_{\R^{2d}} f_0(x, v) \varphi(0, x, v) = 0.
\ena
$$
\end{enumerate}
\end{definition}

The main result of the paper is the following theorem.
\begin{theorem}\label{thm:main}
Suppose that Assumption \ref{ass:psi} holds. Given a $f_0\in \P^2_a(\R^{2d}) \cap L^p(\R^{2d})$ for some $1 < p \leq \infty$ and $\int_{\R^{2d}} f_0(x, v) \Psi(v)dvdx < \infty$, there exists a weak solution $f(x, v, t)$ to \eqref{eq:fracKramers} in the sense of Definition \ref{def:weak}.
\end{theorem}

The proof of Theorem \ref{thm:main} is constructive, that is we will build a converging sequence to a solution of \eqref{eq:fracKramers} from the splitting scheme discussed
above that will be rigorously formulated in Section~\ref{sec: Scheme}.
The proof is based on a series of lemmas and is postponed to Section \ref{sec:proofmain}. 
As a by-product of the analysis, we also construct a discrete variational scheme and obtain its convergence for the kinetic transport equation,
see Theorem~\ref{thm: kinetic transport} in Section~\ref{sec: kinetic transport};  
thus extending the work~\cite{KinderlehrerTudorascu06} to include the kinetic feature. Furthermore, some possible extensions to more complex systems are discussed in Section~\ref{sec: extension}. It is not clear to us how to obtain
the uniqueness and regularity result. The bootstrap argument in~\cite{JKO98} to prove smoothness of weak solutions
(and hence also uniqueness) seems not working for the fractional Laplacian operator due to the lack of a product rule. It should be mentioned that in the recent paper \cite{Lafleche2018}, the author has proved the existence and uniqueness of a solution to the fractional Fokker-Planck equation \eqref{eq:FFP} in some weighted Lebesgue spaces. It would be an interesting problem to generalize \cite{Lafleche2018} to FKFPE. This is to be investigated in  future work. 


\subsection{Organization of the paper}
The rest of the paper is organized as follows. Section~\ref{sec: FHE} summarizes some basic results about the fractional heat equation.
Section~\ref{sec: kinetic transport} studies the kinetic transport equation and its variational formulation. 
The splitting scheme of the paper is formulated explicitly in Section~\ref{sec: Scheme} and some a priori estimates are established for the discrete sequences as well as their time-interpolation. 
The proof of the main result is presented in Section~\ref{sec:proofmain}. Finally, in Section~\ref{sec: extension} we discuss several possible extensions of the analysis to more complex systems.
\subsection{Notation}
Let $\P^2(\R^d)$ be the collection of probability measures on $\R^d$ with finite second moments. Let $\P^2_a(\R^d)$ be the subset of probability measures in $\P^2(\R^d)$ that are absolutely 
continuous with respect to the Lebesgue measure on $\R^d$. For $\mu, \nu \in \P^2(\R^d)$, the 2-Wasserstein distance $W_2(\mu, \nu)$ is defined by 
$$
W_2(\mu, \nu) := \Big(\inf \Big\{\int_{\R^{2d}} |x - y|^2 p(dx, dy) : p \in \P(\mu, \nu)\Big\}\Big)^{\frac{1}{2}}
$$
where $\P(\mu,\nu)$ is the set of probability measures on $\R^{2d}$ with marginals $\mu, \nu$, i.e. $p \in \P(\mu, \nu)$ if and only if
$$
p(A\times \R^d) = \mu(A), \quad p(\R^d \times A) = \nu(A)
$$
hold every Borel set $A\in \R^d$. In the case that $\mu, \nu \in \P^2_a(\R^d)$ with densities $f, g$, we may write $W_2(f, g)$ instead of $W_2(\mu, \nu)$.  

We use the notation $F^\# \mu$ to denote the push-forward of a probability measure $\mu$ on $\R^{2d}$ under map $F$, that is a probability measure on $\R^{2d}$ satisfying for all smooth test function $\varphi$,
\[
\int_{\R^{2d}}\varphi(x,v)\,dF^\#\mu =\int_{\R^{2d}}\varphi(F(x,v))\,d\mu.
\]
\section{The fractional heat equation}
\label{sec: FHE}
This section collects some basic results on the fractional heat equation. We start by defining the fractional heat kernel 
\be\label{eq:phis}
\Phi_s (v, t) := \F^{-1} (e^{-t |\cdot|^{2s}})(v).
\en
Remember that the fractional Laplacian operator in \eqref{eq:fracKramers} is only an operator in $v$-variable. With the fractional heat kernel, the solution to the fractional heat equation \eqref{eq:fhe} with initial condition $f_0(x,v)$ can be expressed as
\be\label{eq:rcon}
f(x, v, t) = \Phi_s (\cdot, t) \ast_v f_0(x, v)
\en
where $\ast_v$ is the convolution operator in $v$-variable. The following elementary result is immediate from the definition of the kernel; see also \cite{AguelBowles2015}. 

\begin{lemma}\label{lem:phi}

\item[(1)] For any $t>0$, $\|\Phi_s(\cdot,t)\|_{L^1(\R^d)} =1$.

\item[(2)] For any $t > 0$ and $p\in (1, \infty)$, $\|\Phi_s (\cdot, t) \ast_v f_0\|_{L^p(\R^{2d})} \leq \|f_0\|_{L^p(\R^{2d})}$.

\item[(3)] $\int_{\R} |v|^2 \Phi_s(v, t)dv = +\infty$ for all $s\in (0,1)$ and $t > 0$.
\end{lemma}

Lemma  \ref{lem:phi}  (3) demonstrates a significant difference between the fractional heat kernel and standard Gaussian kernel, i.e. the former has infinite second moment. The loss of second moment bound may lead to infinite potential energy for example when the potential $\Psi(v) = |v|^2$.  To overcome this issue, it is more convenient to make a renormalisation on the fractional heat kernel. To be more precise, for any $h > 0$, let us denote $\Phi_s^h(v) := \Phi_s(v, h)$ and set $\Phi_{s, R}^h(v) := \Phi_s^h(v) \mathbf{1}_{B_R}(v)$ where $ \mathbf{1}_{B_R}$ is an indicator function of a centred ball of radius $R$. Given a function $f \in \P_a^2(\R^d)$, we can define the renormalised convolution 
\be\label{eq:rcon2}
\bar{f}_{h, R} := \frac{\Phi_{s, R}^h \ast_v f}{\|\Phi_{s, R}^h\|_{L^1(\R^d)}}
\en
It is clear that the new defined convolution satisfies $\bar{f}_{h, R} \gt \Phi_{s}^h \ast_v f$ pointwise. Moreover, we have the following lemma. 

\begin{lemma}\label{lem:fbarF}
Let $f$ be a function such that $F\in C^{1,1}(\R^d) \cap C^2(\R^d)$. 
Suppose that $f\in \P^2_a(\R^{2d})$ and with $\int_{\R^{2d}} f(x, v) F(v) dv dx < \infty$. Then 
\item[(1)] $\bar{f}_{h, R} \in \P^2_a(\R^{2d})$.

\item[(2)]
\be\label{eq:fbarF}
\bea
\int_{\R^{2d}} \bar{f}_{h, R}(x, v) F(v) dx dv &\leq \int_{\R^{2d}} f(x, v) F(v) dx dv \\
&\qquad + \frac{1}{2} \|D^2 F\|_{L^\infty(\R^d)} \frac{\int_{B_R} |w|^2 \Phi_s^h(w)dw}{\int_{B_R} \Phi_s^h(w)dw}.
\ena
\en
\end{lemma}
\begin{proof}
Notice that it suffices to prove part (2) since part (1) follows directly from part (2) by setting $F(v) = |v|^2$. The proof is similar to that of \cite[Lemma 4.1]{AguelBowles2015}, but for completeness we give the proof below. First from the definition of $\bar{f}_{h, R}$, one sees that
$$\bea
\int_{\R^{2d}}\bar{f}_{h, R}(x, v)F(v)\,dxdv =\frac{\int_{\R^{2d}}F(v)\int_{B_R}\Phi_s^h(w)f(x,v-w)dw\, dxdv}{\int_{B_R}\Phi_s^h(w)\,dw}.
\ena
$$
Using change of variable $z=v-w$ and Taylor's expansion, we can write the numerator as
\begin{align}
&\int_{\R^{2d}}F(v)\int_{B_R}\Phi_s^h(w)f(x,v-w)dw\, dxdv\notag
\\&=\int_{\R^{2d}}F(w+z)\int_{B_R}\Phi_s^h(w)f(x,z)dw\, dxdz\notag
\\&=\int_{B_R}\Phi_s^h(w)\int_{\R^{2d}}F(w+z)f(x,z)\,dxdzdw\notag
\\&=\int_{B_R}\Phi_s^h(w)\int_{\R^{2d}}\Big[F(z)+w\cdot \nabla F(z)+\frac{1}{2}w^TD^2F(\xi_{w,z}) w\Big]f(x,z)\,dxdzdw \notag
\\&\leq \int_{B_R}\Phi_s^h(w)\,dw\int_{\R^{2d}}F(z)f(x,z)\,dxdz+\left|\int_{B_R}\Phi_s^h(w)\int_{\R^{2d}}w\cdot\nabla F(z) f(x,z)\,dxdzdw\right|\notag
\\&\qquad +\frac{1}{2}\parallel D^2 F\parallel_{\infty}\,\int_{B_R}|w|^2\Phi_s^h(w)\int_{\R^{2d}}f(x,z)\,dxdz\notag
\\&=\int_{B_R}\Phi_s^h(w)\,dw\int_{\R^{2d}}F(z)f(x,z)\,dxdz+\frac{1}{2}\parallel D^2 F\parallel_{\infty}\,\int_{B_R}|w|^2\Phi_s^h(w).\notag
\end{align}
Note that in the above $\xi_{w,z}$ is an intermediate point between $w$ and $z$ and the term with the modulus vanishes since the kernel $\Phi_s^h$ is symmetric with respect to the origin. 

\end{proof}

The following lemma provides an upper bound for the ratio on the right side of \eqref{eq:fbarF}.
\begin{lemma}\label{lem:ratioconv}
 For any $s\in (0,1]$, there exists a constant $C > 0$ such that 
 \be
 \frac{\int_{B_R} |w|^2 \Phi_s^h(w)dw}{\int_{B_R} \Phi_s^h(w)dw} \leq C (h^{\frac{1}{s}} + h R^{2 -2s})
 \en
 holds for all  $R,h > 0$.
\end{lemma}
\begin{proof}
 This lemma follows directly from a two-sided point-wise estimate on $\Phi_s^h(w)$ as shown in \cite[Proposition 2.1]{AguelBowles2015}. See also equation (16) in \cite{AguelBowles2015}.
\end{proof}

\section{The kinetic transport equation and its variational formulation}
\label{sec: kinetic transport}
\subsection{The minimum acceleration cost}
Consider the kinetic transport equation with initial value $f_0$
\be\label{eq:ktransport}
\bea
&\partial_t f(x, v, t) + v\cdot \nabla_x  f = \div_v(\nabla \Psi(v)f(x, v, t)),\\
& f(x, v, 0) = f_0(x, v).
\ena
\en
We are interested in the variational structure of \eqref{eq:ktransport} which is an interesting problem on its own right. In~\cite{KinderlehrerTudorascu06},
Kinderlehrer and Tudorascu proved that the transport equation $\partial_t f(v,t)=\div_v(\nabla \Psi f)$, which is the
spatially homogeneous version of~\eqref{eq:ktransport}, is a Wasserstein gradient flow of the energy $\int_{\R^d} \Psi f$. 
Their proof is via constructing a discrete variational scheme as in~\cite{JKO98}. However, due to the absence of the entropy term, which is super-linear,
several non-trivial technicalities were introduced to obtain the compactness of the discrete approximations thus establishing the convergence of the scheme. 
For the kinetic transport equation~\eqref{eq:ktransport}, due to the presence of the kinetic term, it is not a Wasserstein gradient flow in the phase
space thus the Wasserstein distance can no longer be used. Therefore to construct a discrete variational scheme for this equation, 
we need a different Kantorovich optimal transportation cost functional. To this end, we will employ the Kantorovich optimal transportation cost functional 
that is associated to the minimal acceleration cost. This cost functional has been used before in~\cite{Hua00, DPZ13a} for the kinetic Fokker-Planck equation
and in~\cite{GW09} for the isentropic Euler system.
We follow the heuristics of defining the minimal acceleration cost as in \cite{GW09}. 
Consider the motion of particle going from position $x$ with velocity $v$ to a new position $x'$ with velocity $v'$, 
within a time interval of length $h$. Suppose that the particle follows a curve $\xi : [0, h] \mapsto \R^d$ such that 
$$
(\xi, \dot{\xi})|_{t=0} = (x, v)\ \text{ and }\ (\xi, \dot{\xi})|_{t=h} = (x', v')
$$
and such that the average acceleration cost along the curve, that is $\frac{1}{h} \int_0^h |\ddot{\xi}(t)|^2 dt$ is minimized. Then the curve is actually a cubic polynomial and the minimal average acceleration cost is given by $C_h(x, v; x', v')/h^2$ where
\be\label{eq:ch}
C_h(x, v; x', v') := |v' - v|^2 + 12\Big|\frac{x' - x}{h} - \frac{v' + v}{2}\Big|.
\en
The Kantorovich functional $\W_h(\mu,\nu)$ associated with the cost function $C_h$,  is defined by, for any $\mu, \nu \in \P^2(\R^{2d})$,
\begin{equation}
\label{eq: Wh distance}
\W_h(\mu,\nu)^2 = \inf_{p\in\P(\mu,\nu)}\int_{\R^{4d}}C_h(x,v;x',v')p(dxdvdx'dv'),
\end{equation}
where $\P(\mu, \nu)$ is the set of all couplings between $\mu$ and $\nu$. It is important to notice that $\W_h$ is not a distance. 
In fact, $\W_h$ is not symmetric in the arguments $\mu,\nu$, due to the asymmetry of the cost function $C_h$. In addition,  $\W_h(\mu,\nu)$ does not vanish when $\mu=\nu$. Instead, we have that
\[
\W_h(\mu,\nu)=0 \quad\Longleftrightarrow\quad \nu = (F_h)_\#\mu,
\]
where $F_h$ is the free transport map defined by
\begin{align}\label{eq:fh}
F_h: \R^d\times\R^d&\to \R^d\times\R^d\notag
\\(x,v)&\mapsto F_h(x,v)=(x+hv, v).
\end{align}
It is also useful to define the map
\begin{align}\label{eq:gh}
G_h: \R^d\times\R^d&\to \R^d\times\R^d\notag
\\(x,v)&\mapsto G_h(x,v)=\left(\sqrt{3}\Big(\frac{2x}{h}-v\Big),v\right).\end{align}
The composition $G_h\circ F_h$ is then given by
\[
(G_h\circ F_h)(x,v)=\left(\sqrt{3}\Big(\frac{2x}{h}+v\Big),v\right).
\]
Although the Kantorovich functional $\W_h(\mu, \nu)$ is not a distance, the next lemma shows that $\W_h$ can be expressed in terms of the usual Wasserstein distance $W_2$. 

\begin{lemma}\cite[Proposition 4.4]{GW09}
\label{lem: aux1}
Let $F_h$ and $G_h$ be given by \eqref{eq:fh} and \eqref{eq:gh} respectively. 
The Kantorovich functional $\W_h$ can be expressed in terms of the 2-Wasserstein distance $W_2$ as
\begin{equation}
\W_h(\mu,\nu)=W_2((G_h\circ F_h)^{\#} \mu, G_h^{\#} \nu)\quad\text{for all}~~\mu,\nu\in\P^2(\R^{2d}). 
\end{equation}
As a consequence, the infimum in \eqref{eq: Wh distance} is attained and thus $\W_h(\mu,\nu)$ is a minimum.
\end{lemma}
 
 \subsection{Variational formulation}
 With $\W_h$ being defined, we want to interpret \eqref{eq:ktransport} as a generalized gradient flow of the potential energy $\int_{\R^{2d}} \Psi(v) f(x,v)dxdv$ with respect to $\W_h$. For doing so, we consider the variational problem
 \be\label{eq:kineticvp}
 \inf_{f\in \P_a^2} \A(f) := \frac{1}{2h} \W_h (f_0, f)^2 + \int_{\R^{2d}} \Psi(v) f(x,v)dxdv.
 \en
 Here $f_0 \in \P_a^2(\R^{2d})$ is an initial probability density with $\int_{\R^{2d}} \Psi(v) f_0(x, v)\ dxdv < \infty$ and $h > 0$ is the time step. 
 The next lemma establishes some properties about the minimizer to \eqref{eq:kineticvp}.

 \begin{lemma}\label{lem:vp}
 \item[(1)] For $h$ being sufficiently small, the variational problem \eqref{eq:kineticvp} has a unique minimizer $f\in \P_a^2(\R^{2d})$.
 
 \item[(2)] Let $h>0$ be small enough such that $\det (I+hD^2(\Psi(v)))\leq 1+\alpha h$ for some fixed $\alpha>\parallel D^2\Psi\parallel_{L^\infty(\R^d)}$. If $f_0\in L^p(\R^{2d})$ for $1<p<\infty$, then
\begin{equation}
\label{eq: ineq minimizer}
\| f\|^p_{L^p(\R^{2d})}\leq  (1 - \alpha h)^{p-1}\|f_0\|^p_{L^p(\R^{2d})}.
\end{equation}
\item[(3)] $f$ satisfies the following Euler-Lagrange equation: for any $\varphi \in C^\infty_c(\R^{2d})$,
\be
\bea
\label{eq: EL}
&\frac{1}{h}\int_{\R^{4d}}\left[(x'-x)\cdot\nabla_{x'}\varphi(x',v')+(v'-v)\cdot\nabla_{v'}\varphi(x',v')\right]P^{*}(dxdvdx'dv')
\\&\quad -\int_{\R^{2d}}v'\cdot\nabla_{x'}\varphi(x',v')f(x',v')dx'dv'\\
& \quad +\int_{\R^{2d}}\nabla_{v'} \Psi(v')\cdot\nabla_{v'}\varphi(x',v')f(x',v')dx'dv'=\Rc,
\ena
\en
where $P^*$ is the optimal coupling in $\W_h(f_0,f)$ and
\begin{equation}
\bea
\label{eq: errorEL}
\Rc & =-\frac{h}{2}\int_{\R^{4d}}\nabla_{v'}\Psi(v')\cdot\nabla_{x'}\varphi(x',v')P^*(dxdvdx'dv')\\
& =-\frac{h}{2}\int_{\R^{2d}}\nabla_{v'}\Psi(v')\cdot\nabla_{x'}\varphi(x',v')f(x',v')\,dx'dv'.
\ena
\end{equation}

 \end{lemma}
 
 \begin{proof}\item[(1)] Thanks to Lemma \ref{lem: aux1}, we can rewrite the functional $\A$ as 
$$
\bea
\A(f) & = \frac{1}{2h} W_2((G_h\circ F_h)^{\#} f_0, (G_h)^{\#} f)^2 + \int_{\R^{2d}} \Psi(v) (G_h)^{\#} f(dx dv)\\
& = \frac{1}{2h} W_2(\tilde{f_0}, \tilde{f})^2 + \int_{\R^{2d}} \Psi(v) \tilde{f}(dx dv) =: \tilde{A}(\tilde{f}),
\ena
$$
where $\tilde{f_0} = (G_h\circ F_h)^{\#} f_0$ and $ \tilde{f} = (G_h)^{\#} f$. According to \cite[Proposition 1]{KinderlehrerTudorascu06} (see also \cite[Proposition 3.1 ]{AguelBowles2015}), the functional $\tilde{A}$ has a unique minimizer, denoted by $\tilde{f}$. Therefore, the problem \eqref{eq:kineticvp} has a unique minimizer $f=(G_h^{-1})^\# \tilde{f}$.
 
 \item[(2)] This follows directly from \cite[Proposition 1]{KinderlehrerTudorascu06} and the fact that if $\tilde{f} = (G_h)^\# f$ then 
 $$
 \|\tilde{f}\|_{L^p(\R^{2d})}^p = \left(\frac{2h}{\sqrt{3}}\right)^{d(p-1)}  \|f\|_{L^p(\R^{2d})}^p.
 $$
 
 \item[(3)]  The derivation of the Euler-Langrange equation for the minimizer $f$ of the variational problem \eqref{eq:kineticvp} follows the now well-established procedure (see e.g.~\cite{JKO98,Hua00}). For the reader's convenience, we sketch the main steps here.  First, we consider the perturbation of $f$ defined by push-forwarding $f$ under the flows
$\phi,\psi\colon[0,\infty)\times\R^{2d}\rightarrow\R^d$:  
\begin{eqnarray*}
&&\frac{\partial\psi_s}{\partial s}=\zeta(\psi_s,\phi_s),~ \frac{\partial\phi_s}{\partial s}=\eta(\psi_s,\phi_s), \\
&& \psi_0(x,v)=x,~\phi_0(x,v)=v,
\end{eqnarray*}
where $\zeta,\eta\in C_0^\infty(\R^{2d},\R^d)$ will be chosen later. Let us denote $\gamma_s$ to be the push forward of $f$ under the flow $(\psi_s,\phi_s)$.
Since $(\psi_0,\phi_0)=\mathrm{Id}$, it follows that $\gamma_0=f$, and an explicit calculation gives
\begin{equation}
\partial_s\gamma_s\big|_{s=0} =-\text{div}_x (f \zeta)-\text{div}_v (f \eta)
\end{equation}
in the sense of distributions.
Second, thanks to the optimality of $f$, we have that $\A(\gamma_s) \geq \A(f)$ for all $\gamma_s$ defined via the flow above. Then the standard variational arguments as in \cite{JKO98,Hua00} leads to the following stationary equation on $f$:
\begin{align}
&\frac{1}{2h}\int_{\R^{4d}}\left[\nabla_{x'}C_h(x,v;x',v')\cdot\zeta(x',v')+\nabla_{v'}C_h(x,v;x',v')\cdot\eta(x',v')\right]P^*(dxdvdx'dv')\nonumber
\\&\qquad+\int_{\R^{2d}}f(x,v)\nabla \Psi(v)\cdot\eta(x,v)dxdv=0,
\label{EuLageqn}
\end{align}
where $\P^*$ is the optimal coupling in the definition of $\W_h(f_0,f)$.
Third, we choose $\zeta$ and $\eta$ with a given $\varphi\in C_0^\infty(\R^{2d},\R)$ as follows
\begin{equation}
\label{xiphi}
\bea
\zeta(x',v')&=-\frac{h^2}{6}\nabla_{x'}\varphi(x',v')+\frac{1}{2}h\nabla_{v'}\varphi(x',v'),\\
\eta(x',v')&=-\frac{1}{2}h\nabla_{x'}\varphi(x',v')+\nabla_{v'}\varphi(x',v').
\ena
\end{equation}
Now from the definition of the cost functional $C_h(x,v;x',v')$ in \eqref{eq:ch}, we have that
\begin{align*}
\nabla_{x'}C_h&=\frac{24}{h}\left(\frac{x'-x}{h}-\frac{v'+v}{2}\right),
\\\nabla_{v'}C_h&=2(v'-v)-12\left(\frac{x'-x}{h}-\frac{v'+v}{2}\right).
\end{align*}
Therefore, together with \eqref{xiphi}, we calculate 
\begin{align*}
&\nabla_{x'}C_h\cdot\zeta+\nabla_{v'}C_h\cdot\eta
\\\qquad&=\frac{24}{h}\left(\frac{x'-x}{h}-\frac{v'+v}{2}\right)\cdot\left(-\frac{h^2}{6}\nabla_{x'}\varphi(x',v')+\frac{1}{2}h\nabla_{v'}\varphi(x',v')\right)\\
& +\left(2(v'-v)-12\left(\frac{x'-x}{h}-\frac{v'+v}{2}\right)\right)\cdot\left(-\frac{1}{2}h\nabla_{x'}\varphi(x',v')+\nabla_{v'}\varphi(x',v')\right)
\\\qquad&=2\left((x'-x)-hv'\right)\cdot\nabla_{x'}\varphi+2(v'-v)\cdot\nabla_{v'}\varphi.
\end{align*}
The Euler-Lagrange equation \eqref{eq: EL} for the minimizer $f$ follows directly by substituting the equation above back into \eqref{EuLageqn}.
 \end{proof}
We now can build up a discrete variational scheme for the kinetic transport equation as follows. Given $f_0 \in \P_a^2(\R^{2d})$ with $\int_{\R^{2d}} \Psi(v) f_0(x, v)\ dxdv < \infty$ and $h > 0$ is the time step. For every integer $k\geq 1$, we define $f_k$ as the minimizer of the minimization problem
 \be\label{eq:kineticvp2}
 \inf_{f\in \P_a^2}\Big\{\frac{1}{2h} \W_h (f_{k-1}, f)^2 + \int_{\R^{2d}} \Psi(v) f(x,v)dxdv\Big\}.
 \en
The following theorem extends the work~\cite{KinderlehrerTudorascu06} to the kinetic transport equation.
 \begin{theorem}
 \label{thm: kinetic transport}
 Suppose that Assumption \ref{ass:psi} holds. Given a $f_0\in \P^2_a(\R^{2d}) \cap L^p(\R^{2d})$ for some
 $1 < p \leq \infty$ and $\int_{\R^{2d}} f_0(x, v) \Psi(v)dvdx < \infty$, there exists a weak solution $f(x, v, t)$ to equation~ \eqref{eq:ktransport} in the sense of Definition \ref{def:weak} but with the fractional Laplacian term removed.
 \end{theorem}
 \begin{proof}
The proof of this theorem follows the same lines as that of the  Theorem~\ref{thm:main}, that is to show that the discrete variational scheme 
\eqref{eq:kineticvp2} above converges to a weak solution of the kinetic transport equation. Since the proof of Theorem~\ref{thm:main} will be carried out in details in Section~\ref{sec:proofmain}, we omit this proof here.
 \end{proof}

 \section{A splitting scheme for  FKFPE}
  \label{sec: Scheme}
  \subsection{Definition of splitting scheme}
  
  As we mentioned in the introduction section, we object to construct an operator splitting scheme for equation \eqref{eq:fracKramers} by continuously
  alternating processes \eqref{eq:fhe} and \eqref{eq:kt}, where the later is 
  approximated by the generalized gradient flow of the potential energy, or equivalently, the density after a short time step $h$ is approximately
  given by the solution to the variational problem \eqref{eq:kineticvp}. However, there is an issue associated with iterating \eqref{eq:fhe} and \eqref{eq:kineticvp}.
  That is, the solution of the fractional heat equation may not have a finite second moment (see Lemma  \ref{lem:phi} (3)). Hence it can not be used as the 
  initial condition in the variational problem \eqref{eq:kineticvp} since the potential energy might be infinite. To around this issue, we define 
  an approximate fractional diffusion process by using the renormalised convolution \eqref{eq:rcon2} based on the truncted fractional heat kernel. 
  To be more precise, given a fixed $N\in \N$, let us consider a uniform partition $0 = t_0 < t_1 < \cdots <t_N = T$ of the time interval $[0, T]$ with $t_k = kh $ and $h = 1/N$.
  With an initial condition $f_h^0 = f_0$, for $n=1, \cdots, N-1$ we iteratively compute the following:
  
  \begin{itemize}
   \item Given a trunction parameter $R > 0$, compute the renormalised convolution
   
   \be\label{eq:step1}
  \bar{f}^n_{h, R}  :=  \frac{\Phi_{s, R}^h \ast_v f_{h, R}^{n-1}}{\|\Phi_{s, R}^h\|_{L^1(\R^d)}}.
   \en
   
   \item Solve for the minimizer $f^{n+1}_{h, R}$ of the problem
   
   \be\label{eq:step2}
   f^{n+1}_{h, R} := \mathrm{argmin}_{\substack{f\in \P_a^2(\R^d)}}  \frac{1}{2h} \W_h (\bar{f}^n_{h, R}, f)^2 + \int_{\R^{2d}} \Psi(v) f(x,v)dxdv.
   \en
  \end{itemize}

  Note that thanks to Lemma \ref{lem:vp} (1)  the minimizer $f^{n+1}_{h, R}$ in \eqref{eq:step2} is well-defined and unique.
  Moreover, it follows from Lemma \ref{lem:vp} (3) that $f^{n+1}_{h, R}$ satisfies the following equation
  \be
  \bea
\label{eq: ELn}
&\int_{\R^{4d}}\left[(x'-x)\cdot\nabla_{x'}\varphi(x',v')+(v'-v)\cdot\nabla_{v'}\varphi(x',v')\right]P^{n+1}_{h,R}(dxdvdx'dv')
\\& = h\int_{\R^{2d}}(v'\cdot\nabla_{x'}\varphi(x',v') - \nabla_{v'} \Psi(v')\cdot\nabla_{v'}\varphi(x',v'))f^{n+1}_{h,R}(x',v')dx'dv' +\Rc^{n+1}_{h,R},
\ena
  \en
where $P^{n+1}_{h,R}$ is the optimal coupling in $\W_h(\bar{f}^{n+1}_{h,R},f^{n+1}_{h,R})$ and 
\begin{equation}
\label{eq: errorELn}
\Rc^{n+1}_{h,R}=\frac{h^2}{2}\int_{\R^{2d}}\nabla_{v}\Psi(v)\cdot\nabla_{x}\varphi(x,v)f^{n+1}_{h,R}(dxdv).
\end{equation}

  With the scheme being defined above, we obtain a discrete approximating sequence $\{f^n_{h, R}\}_{0 \leq n \leq N}$. 
  Below we define a time-interpolation based on  $\{f^n_{h, R}\}$ and our ultimate goal is to prove that this sequence converges to a weak solution of \eqref{eq:fracKramers}.
  
  {\bf Time-interpolation:} We define $f_{h, R}$ by setting 
  \be\label{eq:tintp}
  f_{h, R}(t) := \Phi_s(t - t_n)\ast_v f^n_{h, R} \text{ for } t\in [t_n, t_{n+1}).
  \en
It is clear that by definition $f_{h, R}$ solves the fractional heat equation on every $[t_n, t_{n+1})$ with initial condition $f^n_{h, R}$. Notice also that $f_{h, R}$ is only right-continuous 
in general. For convenience, we also define 
\be\label{eq:tildefn}
\tilde{f}^{n+1}_{h, R} = \lim_{t \uparrow t_{n+1}} f_{h, R}(t).
\en

\subsection{A priori estimates}
In this section, we prove some useful a priori estimates for the discrete-time sequence $\{f^n_{h, R}\}$ as well as for the time-interpolation sequence $\{f_{h, R}(t)\}$. 
We start by proving an upper bound for the sum of the Kantorovich functionals $\W_h(\bar{f}^n_{h, R}, f^n_{h, R})$.

\begin{lemma}\label{lem:sumWh}
Let $\{\overline{f}^{n}_{h,R}\}$ and $\{f^{n}_{h,R}\}$ be the sequences constructed from the splitting scheme. Then there exists a constant $C>0$, independent of $h$ and $R$, such that
\begin{equation}
\label{eq: priori estimate 1}
\sum_{n=1}^N{\W}_h(\overline{f}^n_{h,R},f^n_{h,R})^2\leq
C\Big(h\int_{\R^{2d}}\Psi(v)\,f_0(x,v)\,dxdv+T\parallel D^2 \Psi\parallel_{\infty}(h^{1/2}+ h R^{2-2s})\Big).
\end{equation}
\end{lemma}
\begin{proof}
Since $f^{n}_{h,R}$ minimizes the functional $f\mapsto\frac{1}{2h}\W_h(\overline{f}^n_{h,R},f)+\int_{\R^{2d}}\Psi(v)f(x,v)\,dxdv$, for all $f\in\P_2(\R^{2d})$, we have
\begin{equation*}
\frac{1}{2h}\W_h(\overline{f}^n_{h,R},f^n_{h,R})^2+\int_{\R^{2d}}\Psi(v)f^n_{h,R}\,dxdv\leq \frac{1}{2h}\W_h(\overline{f}^n_{h,R},f)^2+\int_{\R^{2d}}\Psi(v)f(x,v)\,dxdv.
\end{equation*}
In particular, if we set $f=f^\ast := F_h^\# \overline{f}^n_{h,R}$ where $F_h$ is the free transport map defined in \eqref{eq:fh}, then since $\W_h(\overline{f}^n_{h,R},f^*)=0$ we obtain
\begin{align}
\label{eq: Wh1}
\W_h(\overline{f}^n_{h,R},f^n_{h,R})^2 &\leq 2h\left(\int_{\R^{2d}}\Psi(v)f^*(x,v)\,dxdv-\int_{\R^{2d}}\Psi(v)f^n_{h,R}\,dxdv\right)\notag
\\&=2h\left(\int_{\R^{2d}}\Psi(v)\overline{f}^n_{h,R}(x,v)\,dxdv-\int_{\R^{2d}}\Psi(v)f^n_{h,R}\,dxdv\right).
\end{align}
We have also used the fact that the free transport map $F_h$ has unit Jacobian in the last equality. 
According to Lemma \ref{lem:fbarF} (2), we have
\begin{equation}
\label{eq: Wh2}
\bea
\int_{\R^{2d}}\Psi(v) \overline{f}^n_{h,R}(x,v)\,dxdv 
& \leq \int_{\R^{2d}}\Psi(v) f^{n-1}_{h,R}(x,v)\,dxdv\\
& \qquad +\frac{1}{2}\parallel D^2 \Psi\parallel_{\infty}\,\frac{\int_{B_R}|w|^2\Phi_s^h(w)\,dw}{\int_{B_R}\Phi_s^h(w)\,dw}.
\ena
\end{equation}
Substituting \eqref{eq: Wh2} into \eqref{eq: Wh1}, we obtain
\begin{equation*}
\bea
\W_h(\overline{f}^n_{h,R},f^n_{h,R})^2 & \leq 2h\left(\int_{\R^{2d}}\Psi(v) f^{n-1}_{h,R}(x,v)\,dxdv-\int_{\R^{2d}}\Psi(v)f^n_{h,R}\,dxdv\right)\\
& \qquad + h\parallel D^2 \Psi\parallel_{\infty}\,\frac{\int_{B_R}|w|^2\Phi_s^h(w)\,dw}{\int_{B_R}\Phi_s^h(w)\,dw},
\ena
\end{equation*}
from which, by summing over $n$ from $1$ to $N$ we obtain 
\begin{equation}\label{eq: Wh3}
\sum_{n=1}^N\W_h(\overline{f}^n_{h,R},f^n_{h,R})^2\leq 2h\int_{\R^{2d}}\Psi(v) f_0(x,v)\,dxdv +T\parallel D^2 \Psi\parallel_{\infty}\,\frac{\int_{B_R}|w|^2\Phi_s^h(w)\,dw}{\int_{B_R}\Phi_s^h(w)\,dw}.
\end{equation}
Then the desired estimate follows directly from \eqref{eq: Wh3} and Lemma \ref{lem:ratioconv}.
\end{proof}

We also need some second moment bounds on $f$ with respect to variable $v$.   Given a density function $f$, let us set $M_{2,v}(f):= \int |v|^2 f(x, v)dx dv$.  

\begin{lemma}\label{priorbound2-2}There exist positive constants $C, h_0$ such that when $0 < h < h_0$, it holds for any index $i > 0$ that
\begin{equation}\label{eq:m2v}
M_{2,v}(f^i_{h, R}) \leq  M_{2, v}(f^{i-1}_{h, R})  +4\W_h(\overline{f}^{i}_{h,R},f^{i}_{h,R})^2 +C(h^{1/s}+hR^{2-2s}).
\end{equation}
It follows that 
\be\label{eq:2ndmoment}
\bea
 & \max\Big\{ M_{2, v}(f^n_{h, R}),  M_{2, v}(\bar{f}^n_{h, R})\Big\}   \leq  M_{2, v}(f^0) \\
& \qquad \qquad + 4 \sum_{i=1}^n \W_h(\overline{f}^{i}_{h,R},f^{i}_{h,R})^2 + C(n+1) (h^{1/s}+hR^{2-2s}).
\ena
\en
In addition, let $\tilde{P}^i_h$ be the optimal coupling in the definition of $\W_h(\overline{f}^i_{h, R}, f^{i+1}_{h, R})$. Then 
\begin{equation}\label{eq:dm2xv}
\bea
  \int \Big(|x - x'|^2 + |v - v'|^2 \Big)\widetilde{P}^i_h(dxdv dx^{\prime}dv^{\prime}) & \leq C\W_h(\overline{f}^{i}_{h,R},f^{i}_{h,R})^2\\
& \qquad \qquad + Ch^2 \Big(M_{2, v}(\overline{f}^{i}_{h, R}) + M_{2, v}(f^{i}_{h, R})\Big).
\ena
\end{equation}
\begin{proof}
First from the definition of the cost function $C_h$ in \eqref{eq:ch} we have the following inequalities:
\begin{align}
&|v'-v|^2\leq  C_h(x,v;x',v'); \label{eq: v-moment}
\\&|x'-x|^2=h^2\left|\frac{x'-x}{h}-\frac{v'+v}{2}+\frac{v'+v}{2}\right|^2\notag
\\& \leq h^2\left[2\Big|\frac{x'-x}{h}-\frac{v'+v}{2}\Big|^2+\frac{|v' + v|^2 }{2} \right]\notag
\\&\leq h^2\left( \frac{1}{6}C_h(x,v;x',v')+ \frac{|v'|^2 +|v|^2 }{2} \right).\label{eq: x-moment}
\end{align}
Then there exist constants $C, h_0>0$ such that when $h < h_0$, 
\begin{equation}\label{eq:ineqxv}
|x'-x|^2+|v'-v|^2\leq C C_h(x,v;x',v')+h^2(|v'|^2+|v|^2).
\end{equation}
Now for any fixed $i > 0$, we have
\begin{align}
\int_{\R^{2d}}|v|^2f^{i}_{h,R}&= \int_{\R^{4d}}|v'|^2\widetilde P_{i}^h(dxdvdx'dv')\nonumber
\\&\leq \int_{\R^{4d}}|v'-v|^2\widetilde P_{i}^h(dxdvdx'dv')+ \int_{\R^{4d}}|v|^2\widetilde P_{i}^h(dxdvdx'dv')\nonumber
\\&\overset{\eqref{eq: v-moment}}{\leq} 4 \W_h(\overline{f}^{i}_{h,R},f^{i}_{h,R})^2+\int_{\R^{2d}}|v|^2\overline{f}^i_{h,R}\,dxdv\nonumber
\\&\overset{\eqref{eq:fbarF}}{\leq} 4 \W_h(\overline{f}^{i}_{h,R},f^{i}_{h,R})^2+\int_{\R^{2d}}|v|^2 f^{i-1}_{h,R}dxdv +C(h^{1/s}+hR^{2-2s}).\nonumber
\end{align}
This proves \eqref{eq:m2v}. The estimate \eqref{eq:2ndmoment} follows by summing the estimate \eqref{eq:m2v} over the index $i$ from $1$ to $n$ and inequality \eqref{eq:fbarF} with $F(v) = |v|^2$. 
Finally, the estimate \eqref{eq:dm2xv} follows directly from inequality \eqref{eq:ineqxv} and the definition of $\W_h(\overline{f}^{i}_{h,R},f^{i}_{h,R})$.
\end{proof}
\end{lemma}

In the next lemma, we prove a uniform $L^p$-bound for the time-interpolation sequence $\{f_{h, R}\}$.  
\begin{lemma}\label{lem:lp}
  Let $h>0$ be small enough such that $\det (I+hD^2(\Psi(v)))\leq 1+\alpha h$ for some fixed $\alpha>\parallel D^2\Psi\parallel_{L^\infty(\R^d)}$. If $f_0\in L^p(\R^{2d})$ for $1<p<\infty$, then
\begin{equation}
\label{eq:lpbd}
\| f_{h, R}(t)\|^p_{L^p(\R^{2d})}\leq  e^{\alpha T(1-p)}\|f_0\|^p_{L^p(\R^{2d})}.
\end{equation}
\end{lemma}
\begin{proof}
 First, according to Lemma \ref{lem:vp} (2), we have that 
 $$
 \|f^n_{h, R}\|_{L^p(\R^{2d})}^p  \leq (1 - \alpha h)^{p-1}\|\bar{f}^n_{h, R}\|_{L^p(\R^{2d})}^p.
 $$
 In addition, by the definition of $\bar{f}^n_{h, R}$ (see \eqref{eq:step1}) and Young's inequality for convolution,
 $$
 \|\bar{f}^n_{h, R}\|_{L^p(\R^{2d})}^p \leq \|f^{n-1}_{h, R}\|_{L^p(\R^{2d})}^p.
 $$
 This implies that for any $n > 0$, 
 $$
 \|f^n_{h, R}\|_{L^p(\R^{2d})}^p  \leq (1 - \alpha h)^{n(p-1)}\|f_0\|^p_{L^p(\R^{2d})}.
 $$
 Then by the definition of the time-interpolation $f_{h, R}$ in \eqref{eq:tintp}, we have for any $t \in (t_n, t_{n+1})$ that 
 $$
 \bea
 \|f_{h, R}(t)\|^p_{L^p(\R^{2d})} & = \|\Phi_s(t - t_n)\ast_v f^n_{h, R}\|^p_{L^p(\R^{2d})}\\
 & \leq \|f^n_{h, R}\|^p_{L^p(\R^{2d})}\\
 & \leq  (1 - \alpha h)^{n(p-1)}\|f_0\|^p_{L^p(\R^{2d})}\\
 & \leq e^{aT(1-p)}\|f_0\|^p_{L^p(\R^{2d})}.
 \ena
 $$
\end{proof}

\section{Proof of Theorem \ref{thm:main}}\label{sec:proofmain}

\subsection{Approximate equation}
We first show in the next lemma that the time-interpolation $f_{h, R}$ satisfies an approximate equation. 

\begin{lemma}\label{lem:aweqn}
  Let $\varphi\in C_c^\infty([0,T)\times \R^d\times \R^d)$ with time support in $[-T,T]$. Then
\be
\bea
\label{eq:approxeqn}
& \int_0^T\int_{\R^{2d}}f_{h,R}[\partial_t\varphi+v\cdot\nabla_x\varphi-\nabla_v \Psi\cdot\nabla_v\varphi-(-\Delta_v)^s\varphi]\,dxdvdt\\
& \quad \quad +\int_{\R^{2d}}f_0(x,v)\varphi(0,x,v)\,dxdv=\Rc(h,R),
\ena
\en
where $\Rc(h,R) = \sum_{j=1}^4 \Rc_{j}(h, R) + \tilde{\Rc}(h, R)$ and 
\begin{align}
\Rc_{1}(h, R) &=\sum_{n=1}^{N}\int_{\R^{2d}}\varphi(t_{n})(\tilde{f}^{n}_{h,R}-\overline{f}^{n}_{h,R})\,dx\,dv, \label{eq: approx term1}
\\
\Rc_{2}(h, R)& =\sum_{n=1}^{N-1}\int_{t_n}^{t_{n+1}}\int_{\R^{2d}}\Big(\big(v\cdot\nabla_{x}\varphi(t,x,v)-\nabla_{v} F(v)\cdot\nabla_{v}\varphi(t,x,v)\big)\,f_{h,R}(t,x,v)\nonumber
\\& \qquad
-\big(v\cdot\nabla_{x}\varphi(t_{n},x,v)-\nabla_{v} F(v)\cdot\nabla_{v}\varphi(t_{n},x,v)\big)\,f^{n}_{h,R}(x,v)
\Big)\,dxdvdt, \label{eq: approx term2}
\\\Rc_{3}(h, R)& =\int_0^h\int_{\R^{2d}}\Phi_s(t)\ast f_0 \big(v\cdot\nabla_{x}\varphi(t,x,v)-\nabla_{v} F(v)\cdot\nabla_{v}\varphi(t,x,v)\big)\,dxdvdt, \label{eq: approx term3}
\\ \Rc_{4}(h, R)& = \frac{h^2}{2} \sum_{n=1}^N \int_{\R^2}\nabla_v \Psi(v) \cdot \nabla_x \varphi(x, v) f^n_{h, R}(dxdv).
\end{align}
Moreover,
$$
\bea
\tilde{\Rc}(h, R) & \leq \frac{1}{2}\sum_{n=1}^N\|\nabla^2 \varphi(t_n)\|_{\infty} \int_{\R^4} \Big(|x - x'|^2 + |v - v'|^2\Big)P^{n}_{h, R}(dxdvdx'dv').
\ena 
$$
Here $P^{n}_{h, R}$ is the optimal coupling in the definition of $\W_{h}(\bar{f}^n_{h,R}, f^n_{h, R})$.
\end{lemma}

\begin{proof}
 From the definition of $f_{h, R}$ (see \eqref{eq:tintp}) and integration by parts, we obtain that
 \be
\bea\label{eq: conv1}
&\int_{t_n}^{t_{n+1}}\int_{\R^{2d}}f_{h,R}(t)\partial_t\varphi(t)\,dt\,dx\,dv
\\&=\int_{\R^{2d}}(\varphi(t_{n+1})\tilde{f}^{n+1}_{h,R}-\varphi(t_n)f^n_{h,R})\,dx\,dv-\int_{t_n}^{t_{n+1}}\int_{\R^{2d}}\varphi(t)\partial_t f_{h,R}(t)\,dt\,dx\,dv
\\&=\int_{\R^{2d}}(\varphi(t_{n+1})\tilde{f}^{n+1}_{h,R}-\varphi(t_n)f^n_{h,R})\,dx\,dv+\int_{t_n}^{t_{n+1}}\int_{\R^{2d}}\varphi(t)(-\Delta_v)^s f_{h,R}(t)\,dt\,dx\,dv
\\&=\int_{\R^{2d}}(\varphi(t_{n+1})\tilde{f}^{n+1}_{h,R}-\varphi(t_n)f^n_{h,R})\,dx\,dv+\int_{t_n}^{t_{n+1}}\int_{\R^{2d}}f_{h,R}(t)(-\Delta_v)^s\varphi(t)\,dt\,dx\,dv,
\ena
\en
where the second equality holds because $f_{h, R}$ solves the fractional heat equation. 

By adding and subtracting a few tems, we can 
 write the first term on the right hand side of \eqref{eq: conv1} as
 \be
\bea\label{eq:conv2}
&\int_{\R^{2d}}(\varphi(t_{n+1})\tilde{f}^{n+1}_{h,R}-\varphi(t_n)f^n_{h,R})\,dx\,dv
\\&\qquad=\int_{\R^{2d}}(\varphi(t_{n+1})f^{n+1}_{h,R}-\varphi(t_n)f^n_{h,R})\,dx\,dv+\int_{\R^{2d}}\varphi(t_{n+1})(\tilde{f}^{n+1}_{h,R}-f^{n+1}_{h,R})\,dx\,dv
\\&\qquad=\int_{\R^{2d}}(\varphi(t_{n+1})f^{n+1}_{h,R}-\varphi(t_n)f^n_{h,R})\,dx\,dv
+\int_{\R^{2d}}\varphi(t_{n+1})(\tilde{f}^{n+1}_{h,R}-\overline{f}^{n+1}_{h,R})\,dx\,dv
\\&\qquad\qquad+\int_{\R^{2d}}\varphi(t_{n+1})(\overline{f}^{n+1}_{h,R}-f^{n+1}_{h,R})\,dx\,dv.
\ena
\en
Now substituting \eqref{eq:conv2} back into \eqref{eq: conv1} and then summing over index $n$ from $0$ to $N-1$ yields
\be\label{eq:cov4}
\bea
&\int_0^T\int_{\R^{2d}}f_{h,R}(t)\partial_t\varphi(t)\,dt\,dx\,dv
\\&=\sum_{n=0}^{N-1}\int_{t_n}^{t_{n+1}}\int_{\R^{2d}}f_{h,R}(t)\partial_t\varphi(t)\,dt\,dx\,dv
\\&=\sum_{n=0}^{N-1}\Bigg[\int_{t_n}^{t_{n+1}}\int_{\R^{2d}}f_{h,R}(t)(-\Delta_v)^s\varphi(t)\,dt\,dx\,dv\\
& \qquad +\int_{\R^{2d}}(\varphi(t_{n+1})f^{n+1}_{h,R}-\varphi(t_n)f^n_{h,R})\,dx\,dv
+\int_{\R^{2d}}\varphi(t_{n+1})(\tilde{f}^{n+1}_{h,R}-\overline{f}^{n+1}_{h,R})\,dx\,dv\\
& \qquad +\int_{\R^{2d}}\varphi(t_{n+1})(\overline{f}^{n+1}_{h,R}-f^{n+1}_{h,R})\,dx\,dv\Bigg]
\\&\qquad=\int_0^T\int_{\R^{2d}}f_{h,R}(t)(-\Delta_v)^s\varphi(t)\,dt\,dx\,dv-\int_{\R^{2d}}\varphi(0)f_0(x,v)\,dxdv
\\&\qquad\qquad+\sum_{n=1}^{N}\int_{\R^{2d}}\varphi(t_{n})(\tilde{f}^{n}_{h,R}-\overline{f}^{n}_{h,R})\,dx\,dv+\sum_{n=1}^{N}\int_{\R^{2d}}\varphi(t_{n})(\overline{f}^{n}_{h,R}-f^{n}_{h,R})\,dx\,dv.
\ena
\en
In the above we also used the fact that $\varphi$ is compactly supported in $(-T, T)$ so that $\varphi(t_{N}) = 0$.
Let $P^{n}_{h,R}(dxdvdx'dv')$ be the optimal coupling in $\W_h(\bar{f}^{n}_{h,R},f^{n}_{h,R})$. Then it is easy to see that
\be
\bea\label{timederivativeappro}
&\int_{\R^{2d}}\big[f^{n}_{h,R}-\overline{f}_{h,R}^{n}\big]\,\varphi(t_{n})dxdv \\
&=\int_{\R^{2d}}f^{n}_{h,R}\varphi(t_{n},x',v')dx'dv'-\int_{\R^{2d}}\overline{f}_{h,R}^{n}(x,v)\varphi(t_{n},x,v)dxdv
\\
&=\int_{\R^{4d}}\big[\varphi(t_{n},x',v')-\varphi(t_{n},x,v)\big]P^{n}_{h,R}(dxdvdx'dv')\,
\\&=\int_{\R^{4d}}\big[(x'-x)\cdot\nabla_{x'}\varphi(t_{n},x',v')+(v'-v)\cdot\nabla_{v'}\varphi(t_{n},x',v')\big]P^{n}_{h,R}(dxdvdx'dv')\\
& \qquad \qquad +\varepsilon_{n},
\ena
\en
where we have used Taylor expansion in the last equality and the error term $\varepsilon_{n}$ can be bounded as
\be
|\varepsilon_{n}|\leq \frac{1}{2} \parallel\nabla^2\varphi(t_{n})\parallel_{\infty}\int_{\R^{4d}}\big[|x'-x|^2+|v'-v|^2\big]\,P^{n}_{h,R}(dxdvdx'dv').
\en
In view of \eqref{eq: ELn}, \eqref{eq: errorELn} and \eqref{timederivativeappro}, we have that
\be\label{eq: approx1}
\bea
& \int_{\R^{2d}}[f^{n}_{h,R}(x,v)-\overline{f}_{h,R}^{n}(x,v)]\,\varphi(t_{n},x,v)dxdv \\
& = h \,\int_{\R^{2d}}\big[v\cdot\nabla_{x}\varphi(t_{n},x,v)-\nabla_{v} \Psi(v)\cdot\nabla_{v}\varphi(t_{n},x,v)\big]\,f^{n}_{h,R}(x,v)\,dxdv\\
& \qquad +  \frac{h^2}{2}\int_{\R^{2d}} \nabla_{v} \Psi(v) \cdot \nabla_x \varphi(t_n, x, v)f^{n}_{h,R}(dxdv) + \varepsilon_n
\ena
\en
and that 
\be\bea 
\label{eq: approx1}
& \int_{\R^{2d}}[f^{n}_{h,R}(x,v)-\overline{f}_{h,R}^{n}(x,v)]\,\varphi(t_{n},x,v)dxdv \\
& =h\,\int_{\R^{2d}}\big[v\cdot\nabla_{x}\varphi(t_{n},x,v)-\nabla_{v} \Psi(v)\cdot\nabla_{v}\varphi(t_{n},x,v)\big]\,f^{n}_{h,R}(x,v)\,dxdv\\
& \qquad  + \frac{h^2}{2}\int_{\R^{2d}}\nabla_{v}\Psi(v)\cdot\nabla_{x} \varphi(t_n, x,v) f^{n}_{h,R}(dxdv) + \varepsilon_n.
\ena
\en
As a result the  last term on the right-hand side of 
\eqref{eq:cov4} can be written as
\be\label{eq:approx2}
\bea
&\sum_{n=1}^{N}\int_{\R^{2d}}[\overline{f}_{h,R}^{n}(x,v)-f^{n}_{h,R}(x,v)]\,\varphi(t_{n},x,v)dxdv\\
& \qquad = -h\,\sum_{n=1}^{N}\int_{\R^{2d}}\big[v\cdot\nabla_{x}\varphi(t_{n},x,v)-\nabla_{v} \Psi(v)\cdot\nabla_{v}\varphi(t_{n},x,v)\big]\,f^{n}_{h,R}(x,v)\,dxdv\\
& \qquad -\frac{h^2}{2} \sum_{n=1}^{N}\int_{\R^{2d}}\nabla_{v}\Psi(v) \cdot \nabla_{x} \varphi(t_n, x,v) f^{n}_{h,R}(dxdv) - \sum_{n=1}^{N} \varepsilon_n.
\ena
\en
Now using again the fact that $\varphi(t_N) = 0$, we rewrite the first term on the right side of \eqref{eq:approx2} as follows
\be\label{eq:approx3}
\bea
&-h\,\sum_{n=1}^{N}\int_{\R^{2d}}\big[v\cdot\nabla_{x}\varphi(t_{n},x,v)-\nabla_{v} \Psi(v)\cdot\nabla_{v}\varphi(t_{n},x,v)\big]\,f^{n}_{h,R}(x,v)\,dxdv\\
& = -\sum_{n=1}^{N-1}\int_{t_n}^{t_{n+1}}\,\int_{\R^{2d}}\big[v\cdot\nabla_{x}\varphi(t_{n},x,v)-\nabla_{v} \Psi(v)\cdot\nabla_{v}\varphi(t_{n},x,v)\big]\,f^{n}_{h,R}(x,v)\,dxdvdt\\
& = -\int_{0}^{T}\,\int_{\R^{2d}}\big[v\cdot\nabla_{x}\varphi(t,x,v)-\nabla_{v} \Psi(v)\cdot\nabla_{v}\varphi(t,x,v)\big]\,f_{h,R}(x,v)\,dxdvdt\\
& \quad+ \sum_{n=1}^{N-1}\int_{t_n}^{t_{n+1}}\,\int_{\R^{2d}}\Bigg( \big[v\cdot\nabla_{x}\varphi(t,x,v)-\nabla_{v} \Psi(v)\cdot\nabla_{v}\varphi(t,x,v)\big]\,f_{h,R}(x,v)\\
& \qquad \qquad \qquad  - \big[v\cdot\nabla_{x}\varphi(t_{n},x,v)-\nabla_{v} \Psi(v)\cdot\nabla_{v}\varphi(t_{n},x,v)\big]\,f^{n}_{h,R}(x,v)\Bigg)\,dxdvdt\\
& \quad+\int_0^h\int_{\R^{2d}}\Phi_s(t)\ast f_0 \big[v\cdot\nabla_{x}\varphi(t,x,v)-\nabla_{v} \Psi(v)\cdot\nabla_{v}\varphi(t,x,v)\big]\,dxdvdt.
\ena
\en
Therefore the lemma follows by combining \eqref{eq:cov4}, \eqref{eq:approx2} and \eqref{eq:approx3}.
 \end{proof}
 \subsection{Passing to the limit}
 Now we set the truncation parameter $R = h^{-1/2}$ and define 
 \be\label{eq:fh2}
 f_{h}(t):= f_{h, h^{-1/2}}(t),  t\in [0, T].
 \en Our aim is to prove that $f_h$ converges to a weak solution of \eqref{eq:fracKramers}.
 To this end, we first show that the residual term in the last lemma goes to zero when $h \gt 0$.
 
 \begin{lemma}  Let $f_0$ be a non-negative function such that $f_0\in \P^2_a(\R^{2d})$ and $\int_{\R^{2d}} f_0(x, v) \Psi(v)dvdx < \infty$. Then as $h \gt 0$, we have that 
  \be\label{eq:ineqR}
  |\Rc(h, h^{-1/2})| \leq C(h^2 + h + h^s + h^{1/s}) \gt 0.
  \en
 \end{lemma}
 
 \begin{proof}
 The proof follows closely the proof of Lemma 5.3 of \cite{AguelBowles2015}. In particular, by using the same arguments there, we can first obtain the following estimates 
  $$
  \bea
  \Rc_1(h, R) & \leq C T \sup_{t \in [0, T]}\|\varphi(t)\|_{\infty} R^{-2s},\\
  \Rc_2(h, R) & \leq  \frac{Th}{2}\sup_{t \in [0, T]} \|v\cdot\nabla_{x}\, \partial_t \varphi(t,x,v)-\nabla_{v} \Psi(v)\cdot\nabla_{v}\, \partial_t\varphi(t,x,v)\|_{\infty}\\
  & + \frac{Th}{2}\sup_{t \in [0, T]} \left\|(-\Delta)^s \Big(v\cdot\nabla_{x}\varphi(t,x,v)-\nabla_{v} \Psi(v)\cdot\nabla_{v}\varphi(t,x,v)\Big)\right\|_{\infty},\\  
   \Rc_3(h, R) & \leq h \sup_{t \in [0, T]} \|v\cdot\nabla_{x} \varphi(t,x,v)-\nabla_{v} \Psi(v)\cdot\nabla_{v}\varphi(t,x,v)\|_{\infty}.
  \ena
  $$
  Notice that the supreme norms appearing in the above are finite since $\varphi \in C^\infty_0((-T\times T)\times \R^{2d})$ and $\Psi\in C^{1,1}\cap C^{2,1}(\R^d)$. 
  Next, we can bound $\Rc_4(h,R)$ as
  $$
  \Rc_4(h, R)  \leq  \frac{Th}{2}\sup_{t \in [0, T]} \|\nabla_v \Psi(v) \cdot \nabla_x \varphi(t,x,v)\|_{\infty}.
  $$
  In addition, thanks to inequality \eqref{eq:dm2xv} and Lemma \ref{lem:sumWh}, the error term $\tilde{\Rc}$ can be bounded as follows
  $$\bea
  \tilde{\Rc}(h, R) & \leq C\sum_{n=1}^N \W_h(\bar{f}^n_{h, R}, f^n_{h, R})^2 + C h^2 \sum_{n=1}^N \Big(M_{2, v}(\overline{f}^{n}_{h, R}) + M_{2, v}(f^{n}_{h, R})\Big)\\
  &\leq C(1 + h^2)\sum_{n=1}^N \W_h(\bar{f}^n_{h, R}, f^n_{h, R})^2 + C h^2 M_{2, v}(f^0) \\
  & \qquad  + C(N+1)Nh^2 (h^{1/s} + h R^{2-2s})\\
  & \leq C \Big(h\int_{\R^{2d}}\Psi(v) f_0(x,v)\,dxdv + T \|D^2 \Psi\|_\infty (h^{1/s} + h R^{2- 2s}) \Big)\\
  & \qquad+ C h^2 M_{2, v}(f^0) + C(T+1)T (h^{1/s} + h R^{2-2s}).
  \ena
  $$
  Finally, the desired estimate \eqref{eq:ineqR} follows by combining the above estimates and by setting $R = h^{-1/2}$. 
 \end{proof}

Now we are ready to prove the main Theorem \ref{thm:main}.

\begin{proof}[Proof of Theorem \ref{thm:main}]
First, thanks to Lemma \ref{lem:lp} and the assumption that $f_0\in L^p(\R^{2d})$ for some $1<p < \infty$, the constructed time-interpolation $\{f_{h}\}$ in \eqref{eq:fh2} is uniformly bounded in $L^p(\R^{2d}\times (0,T))$. Therefore there exists a $f\in L^p(\R^{2d}\times (0,T))$ such that $f_h \wgt h$ in $ L^p(\R^{2d}\times (0,T))$.  In view of equation \eqref{eq:approxeqn} of Lemma \ref{lem:aweqn}, and by using the fact that $\partial_t \varphi  +v\cdot \nabla_x \varphi - \nabla_v \Psi \cdot \nabla_v \varphi - (-\Delta_v)^s \varphi \in L^{p'}(\R^{2d}\times (0,T))$, we obtain by letting $h\gt 0$ that
$$
\bea
& \int_0^T\int_{\R^{2d}}f [\partial_t\varphi+v\cdot\nabla_x\varphi-\nabla_v \Psi\cdot\nabla_v\varphi-(-\Delta_v)^s\varphi]\,dxdvdt\\
& \qquad \qquad +\int_{\R^{2d}}f_0(x,v)\varphi(0,x,v)\,dxdv = 0.
\ena
$$
\end{proof}

\begin{remark}
By using the similar technique as in the proof of Lemma 5.8 of \cite{AguelBowles2015}, one can show that the weak solution $f$ of \eqref{eq:fracKramers} is indeed a probability density for every $t \in (0,T)$, i.e. $\int_{\R^{2d}} f(t,x ,v)dxdv = \int_{\R^{2d}} f_0(x, v)dxdv  = 1$.
\end{remark}

\section{Possible extensions to more complex systems}
\label{sec: extension}
With suitable adaptations, it should be possible, in principle, to extend the analysis of the present work to deal with more complex systems. Below we briefly discuss two such systems.
\subsection{ FKFPE with external force fields}
When an external force field, which is assumed to be conservative, is present, the SDE \eqref{SDE} becomes
\begin{equation}
\bea
& \frac{d X_t}{d t} = V_t,\\
& \frac{d V_t}{d t} = -\nabla U(X_t)- \nabla \Psi (V_t) + L_t^s,
\ena
\label{SDE2}
\end{equation}
where $U:\R^d\to\R$ is the external potential. The corresponding FKFPE \eqref{eq:fracKramers} is then given by
\begin{equation}
\begin{cases}
\partial_t f+v\cdot\nabla_x f=\div_v (\nabla V(x)f)+\div_v(\nabla \Psi(v)f)-(-\Delta_v)^s f ~~ \text{in} ~~ \R^d\times \R^d\times (0,\infty),\\
f(x,v,0)=f_0(x,v)~~\text{in}~ \R^d\times \R^d.
\end{cases}
\label{eq:fracKramers ext}
\end{equation}
One can view \eqref{SDE2} as a dissipative (frictional and stochastic noise) perturbation of the classical Hamiltonian
\begin{equation*}
\bea
& \frac{d X_t}{d t} = V_t,\\
& \frac{d V_t}{d t} = -\nabla U(X_t).
\ena
\end{equation*}
Thus  FKFPE \eqref{eq:fracKramers ext} contains both conservative and dissipative effects. To construct an approximation scheme for it, instead of the minimal acceleration cost function \eqref{eq:ch}, one would use the following minimal Hamiltonian cost function which has been introduced in \cite{DPZ13a} for the development of a variational scheme for the classical Kramers equation:
\begin{multline}
\label{def:widetildeCh}
\widetilde{C}_h(x,v;x',v'):= h\inf \bigg\{\int_0^h\bigl|\ddot{\xi}(t)+ \nabla V(\xi(t))\bigr|^2\,dt: \xi\in C^1([0,h],\R^d)
~~ \text{such that}\\~~ 
(\xi,\dot\xi)(0)=(x,v),\ (\xi,\dot\xi)(h)=(x',v')\bigg\}.
\end{multline}
Physically, the optimal value $C_h(x,v;x',v')$ measures the least deviation from a Hamiltonian flow that connects $(x,v)$ and $(x',v')$ in the time interval $[0,h]$. 

Under the assumption that $U\in C^2(\R^d)$ with $\|\nabla^2 U\|\leq C$ and using the properties of the cost function $\widetilde C_h$ established in \cite{DPZ13a} we expect that the splitting scheme \eqref{eq:step1}-\eqref{eq:step2}, where in \eqref{eq: Wh distance} the Kantorovich optimal cost functional $C_h$ is replaced by $\widetilde C_h$, can be proved to converge to a weak solution of  FKFPE \eqref{eq:fracKramers ext}.
\subsection{A multi-component FKFPE equation}
The second system is an extension of  FKFPE \eqref{eq:fracKramers} on the phase space $(x,v)\in \R^{2d}$ to a multi-component FKFPE
on the space $\x=(x_1,\ldots, x_n)\in \R^{nd}$
\begin{equation}
\begin{cases}
\label{eq: Kramers eqn ext2}
\partial_t f+\sum_{i=2}^{n}x_{i}\cdot\nabla_{x_{i-1}}f=\div_{x_n}(\nabla V(x_n)f)-(-\Delta_{x_n})^s f \quad \text{in} \quad \R^{nd}\times (0,\infty),\\
f(x_1,\ldots,x_n,0)=f_0(x_1,\ldots,x_n)\quad \text{in} \quad \R^{nd}.
\end{cases}
\end{equation}
Equation \eqref{eq: Kramers eqn ext2} with $n>2$ and $s=1$ has been studied extensively in the mathematical literature and has found many applications in different fields. For instance, it has been used as a simplified model of a finite Markovian approximation for the generalised Langevin dynamics \cite{OP11, Duong15NA} or a model of a harmonic chains of oscillators that arises in the context of non-equilibrium statistical mechanics \cite{EH00,BL08, DelarueMenozzi10}. It has also appeared in mathematical finance \cite{Pascucci2005}.  Regularity properties of solutions  to equation~\eqref{eq: Kramers eqn ext2} with $s \in (0,1]$ has been investigated recently \cite{HMP2017TMP,ChenZhang2017TMP, ChenZhang2017TMP2}.

To construct an approximation scheme for equation~\eqref{eq: Kramers eqn ext2}, instead of the minimal acceleration cost function \eqref{eq:ch}, one would use the so-called mean squared derivative cost function
\begin{equation*}
\bar{C}_{n,h}(x_{1},x_{2},\ldots,x_{n};y_{1},y_{2},\ldots,y_{n}):=h\inf\limits_{\xi}\int_0^h|{\xi}^{(n)}(t)|^{2}\,dt,
\end{equation*}
where $\mathbf{x}=(x_{1},\ldots,x_{n})\in\R^{nd},~\mathbf{y}=(y_{1},\ldots,y_{n})\in\R^{nd}$, and
the infimum is taken over all curves $\xi\in C^{n}([0,h],\R^d)$ that satisfy the boundary conditions
\begin{equation*}
(\xi,\dot{\xi},\ldots,\xi^{(n-1)})(0)=(x_{1},x_{2},\ldots,x_{n})\quad\text{and}\quad (\xi,\dot{\xi},\ldots,\xi^{(n-1)})(h)=(y_{1},y_{2},\ldots,y_{n}).
\end{equation*}
Several properties including an explicit representation of the mean squared derivative cost function has been studied in \cite{DuongTran2017a} and a variational formulation using this cost function for equation \eqref{eq: Kramers eqn ext2} with and $s=1$ has been developed recently in~\cite{DuongTran2018}. 

Using the properties of the cost function $\bar C_{n,h}$ established in \cite{DuongTran2017a} it should be possible, in principle, to adapt the analysis of the present paper to show that, under suitable assumptions, the splitting scheme \eqref{eq:step1}-\eqref{eq:step2} with $C_h$ being substituted by $\bar C_{n,h}$, converges to a weak solution of the multi-component FKFPE \eqref{eq: Kramers eqn ext2}.

\section*{Acknowledgments} This work was partially done during the authors' stay at Warwick Mathematics Institute. The authors thank WMI for its 
great academic and administrative support. M. H. Duong was also supported by ERC Starting Grant 335120. 

\bibliographystyle{abbrv}
\bibliography{ref}

\end{document}